\documentclass{article}
\usepackage[utf8]{inputenc}
\usepackage[spanish,english]{babel}
\usepackage[pdfborder={0 0 0}]{hyperref}
\usepackage{subcaption}
\usepackage{graphicx}

\usepackage{algorithm}
\usepackage{algpseudocode}
\usepackage{listings}

\usepackage{amsmath}
\usepackage{xcolor, soul}
\sethlcolor{green}
\usepackage{amssymb}
\usepackage{mathrsfs}
\usepackage{bm}  
\usepackage{amsthm}
\usepackage{tikz-cd}
\usepackage{tikz}
\usetikzlibrary{cd}
\usetikzlibrary{cd}
\newtheorem{thm}{Theorem}
\newtheorem{lem}[thm]{Lemma}
\newtheorem{cor}[thm]{Corollary}
\newtheorem*{cor*}{Corolario}
\newtheorem*{thm*}{Theorem}

\newtheorem{ex}[thm]{Example}

\theoremstyle{definition}
\newtheorem{definition}{Definition}
\newtheorem{prop}{Proposition}
\newtheorem*{prop*}{Proposition}

\theoremstyle{remark}
\newtheorem{rem}{\textbf{Remark}}
\theoremstyle{remark}

\DeclareMathOperator{\Gr}{Gr}

\DeclareMathOperator{\supp}{supp}

\newcommand{\C}{\mathbb{C}}

\newcommand{\R}{\mathbb{R}}
\newcommand{\Z}{\mathbb{Z}}
\newcommand{\N}{\mathbb{N}}

\usepackage[a4paper,top=3cm,bottom=2cm,left=3cm,right=3cm,marginparwidth=1.75cm]{geometry}

\renewcommand{\P}{\mathbb{P}}

\DeclareMathOperator{\In}{in}
\DeclareMathOperator{\trop}{trop}
\newcommand{\I}{\mathcal{I}}
\newcommand{\V}{\mathfrak{v}}
\renewcommand{\a}{\mathbf{a}}
\renewcommand{\b}{\mathbf{b}}
\newcommand{\x}{\mathbf{x}}

\newcommand{\w}{\mathbf{w}}
\newcommand{\xa}{{\x}^{\a}}
\newcommand{\xb}{{\x}^{\b}}
\renewcommand{\sl}{\mathfrak{sl}}

\newcommand{\p}{\mathbf{p}}

\DeclareMathOperator{\Ht}{ht}
\newcommand{\U}{\mathbf{u}}

\DeclareMathOperator{\Mon}{Mon}
\DeclareMathOperator{\im}{im}

\title{Birational sequences for the Grassmannian $\Gr\left(3,n\right)$}
\author{Joaquin Torres Henestroza}
\date{2025}

\begin{document}

\maketitle
\begin{abstract}
    Following the ideas of Bossinger and Fang, Fourier, and Littelman, we study \emph{iterated sequences} for the Grassmannian $\Gr\left(3,n\right)$ as a special class of birational sequences. For each iterated sequence, there is a weighting matrix $M_S$ corresponding to a valuation on the rational coordinate ring and we show that the initial form of a Plücker relation $\In_{M_S}\left(R_{I,J}\right) $ is binomial. We show that, in some cases, the cones $C_S$ in the tropical Grassmannian that satisfy $\In_{M_S}\left(\I_{3,n}\right)=\In_{C_S}\left(\I_{3,n}\right)$ only depend on the first two indices used in each iteration. In the case of $\Gr\left(3,6\right)$, these cones are obtained computationally and are classified up to automorphism induced by the symmetric group $S_6$.  
\end{abstract}

\tableofcontents

\vspace{0,5cm} 

In the context of algebraic geometry, toric varieties are especially popular for their connections between their geometric properties and combinatorial objects associated to them. These connections can be extended to certain projective varieties if there exists a \emph{toric degeneration}. Formally, given a (projective) variety $X$, a \emph{toric degeneration} is a flat family $\pi:\mathcal{X}\rightarrow \mathbb{A}_{\C}^1$ such that $\pi^{-1}\left(t\right)$ is isomorphic to $X$ for $t\neq 0$, and $\pi^{-1}\left(0\right)$ is a toric variety. If a variety $X$ admits said degeneration, then both $X$ and the toric variety $\pi^{-1}\left(0\right)$ share common properties (cf. \cite{bossinger2023surveytoricdegenerationsprojective}), thus the next questions to ask are: for which varieties is it possible to construct such degenerations and how can these degenerations actually be constructed? 

In pursuit of partial answers to these questions, the goal of this paper is the construction of toric degenerations which arise as Gröbner degenerations of the Grassmannian $\Gr\left(3,n\right)$. In order to do this, the usual approach is computing initial ideals of the Plücker ideal $\I_{3,n}$ with respect to weight vectors in the tropical Grassmannian $\trop\left(\Gr(3,n)\right)$. Thus, the problem of finding degenerations reduces, in this case, to finding constructions of weight vectors. In the case of the Grasmannian $\Gr\left(3,6\right)$ such constructions can be done, for example, using plabic graphs (cf. \cite{bossinger2016toricdegenerationsgr2ngr36}), matching fields (cf. \cite{Mohammadi_2019}), or by considering \emph{birational sequences}: for $n\geq 4$, let $\Phi$ denote the $\mathbf{A}_{n-1}$ type root system and for each positive root $\beta\in \Phi^+$, consider the one-parameter root subgroup $U_{\beta}\subset U^+$, then a sequence of positive roots $S= \left(\beta_1,\dots,\beta_N\right)$ is a \emph{birational sequence} for $\Gr\left(3,n\right)$ if the multiplication map
\begin{equation*}
    \psi_S := \operatorname{mult}:U_{\beta_N}\times \dots \times U_{\beta_1} \rightarrow U^+
\end{equation*} has image birational to $\Gr\left(3,n\right)$ (see Definition \ref{birseqs}). This last approach, introduced by Fang, Fourier, and Littelman in \cite{FANG2017107} and for the case of a Grassmannian by Bossinger (cf. \cite{Bossinger:2021}), allows to give coordinates on the Grassmannian that gives rise to a valuation on the rational coordinate ring $\C\left(\Gr(3,n)\right)\setminus \lbrace 0\rbrace$, and thus to the weighting matrix $M_S$. By \cite[Theorem 1]{Bossinger:2021}, there is a weight vector $w_S\in \trop\left(\Gr(3,n)\right)$ such that the initial ideal with respect to $M_S$ and the initial ideal with respect to $w_S$ coincide: $\In_{w_S}\left(\I_{3,n}\right) = \In_{M_S}\left(\I_{3,n}\right)$. 

In \cite[Definition 11]{Bossinger:2021}, the special class of \emph{iterated sequences} is introduced. This special kind of birational sequences have the advantage that they are constructed inductively, and therefore the calculations can be implemented computationally. This advantage is used for arbitrary Grassmannians $\Gr(3,n)$, first in Lemma \ref{homogeneidad} and Corollary \ref{only-depends-on-lex}. A direct consequence of these results is Proposition \ref{binomial}. 

\begin{prop*}
    Let $S$ be an iterated sequence for $\Gr\left(3,n\right)$ and $R_{I,J}\neq 0$ a Plücker relation (see Equation \ref{PR}). Then its initial form $\In_{M_S}\left(R_{I,J}\right)$ is binomial.
\end{prop*}

The construction of weight vectors in the tropical Grassmannian relies on the construction of \emph{order preserving projections} (see Section \ref{sec3}). Since order preserving projections must preserve the inequalities obtained when computing the initial forms of Plücker relations (with respect to a weighting matrix $M_S$), a first step in the classification of initial ideals with respect to weighting matrices is given by Propositions  \ref{perm-equality} and \ref{change-equality}.

\begin{prop*}
    Let $S_1$ and $S_2$ be two iterated sequences for $\Gr\left(3,n\right)$. Assume that the first two indices used in each iteration coincide. Given a projection $e_{S_1}:\Z^{3(n-3)}\rightarrow \Z$ that preserves inequalities with respect to $S_1$, there exists a projection $e_{S_2}:\Z^{3(n-3)}\rightarrow \Z$ that preserves inequalities with respect to $S_2$, and satisfies
    \begin{equation*}
        e_{S_1}M_{S_1} = e_{S_2}M_{S_2}. 
    \end{equation*}
\end{prop*}

In the $\Gr\left(3,6\right)$ case, the computations are done by implementing Algorithm \ref{algorithm} to find all inequalities arising from the initial forms of Plücker relations and \verb|polymake| (see \cite{polymake:2000}) to compute the weight vectors; then, the initial ideals are computed  and classified, up to the action of $S_6$, using \verb|macaulay2| (see \cite{M2}). The results are the contents of Theorems \ref{prim-class}, \ref{second_class}, and Corollary \ref{class_toric}, briefly stated below.

%%%%5
% Agregar aquí lo de las graficas y eso
%In \cite[Remark 1]{Bossinger:2021} it is summarized that the toric degenerations for $\Gr\left(2,n\right)$ by considering vectors $w\in \trop\left(\Gr(2,n)\right)$, the degenerations by considering plabic graphs, and the degenerations by considering \emph{iterated} sequences coincide. By \cite[Proposition 5.5]{speyer2003tropicalgrassmannian}, there exist vectors $w\in \trop\left(3,6\right)$ such that the initial ideal $\In_{w}\left(\I_{3,6}\right)$ is not prime, and thus, it does not induce a toric degeneration. Furthermore, for certain plabic graphs for $\Gr\left(3,6\right)$, it is the case that the ideal is not toric, as is mentioned in \cite[$\S7$]{bossinger2016toricdegenerationsgr2ngr36}. Therefore, the aforementioned equivalence between degenerations induced by 

\begin{thm*}
    Let $\mathcal{S}_{3,6}:=\left\lbrace \textrm{iterated sequences for }\Gr(3,6)\right\rbrace$. The map $S\mapsto \In_{M_S}\left(\I_{3,6}\right)$ is a map
    \begin{equation*}
        \mathcal{S}_{3,6}\rightarrow \left\lbrace \textrm{initial toric ideals of }\Gr(3,6)\right\rbrace.
    \end{equation*} The image has cardinality $\#\left\lbrace \In_{M_S}(\I_{3,6}):S\in \mathcal{S}_{3,6}\right\rbrace=240$. Under the action of $S_6\leq \operatorname{Aut}(\I_{3,6})$, there exist in $\left\lbrace \In_{M_S}(\I_{3,6}):S\in \mathcal{S}_{3,6}\right\rbrace$ 4 different equivalence classes $O_1,O_2,O_3,$ and $O_4$, listed with its cardinalities and its corresponding isomorphism classes
\begin{table}[H]
    \centering
    \begin{tabular}{|c|c|c|}
        \hline
        Orbit & $\#$ & Isomorphism class \\
        \hline
        $O_1\cap\mathbb{I}_{3,6}$ & $48$ & $ EEFF1 $\\
        \hline
        $O_2\cap\mathbb{I}_{3,6}$ & $48$ & $EFFG$ \\
        \hline 
        $O_3\cap\mathbb{I}_{3,6}$ & $48$ & $EEFF2$\\
        \hline
        $O_4\cap\mathbb{I}_{3,6}$ & $96$ & $EEFG$\\
        \hline
    \end{tabular}
    \caption*{Classification of the $G$-orbits in $\mathbb{I}_{3,6}$}
\end{table}

    Furthermore, every iterated sequence induces a toric degeneration of $\Gr\left(3,6\right)$, and these are classified according to the isomorphism classes mentioned above.
\end{thm*}

In order to present and prove these results, the paper is divided as follows: in Section \ref{sec1}, the notation and fundamental definitions and results are introduced; in Section \ref{sec2} the birational and iterated sequences are defined, the lowest-term valuation is defined, a proof of Lemma \ref{homogeneidad} is given and Algorithm \ref{algorithm} is described; finally, in Section \ref{sec3} the main results are presented.

\textbf{Acknowledgements:} The results presented in this paper are part of my BSc thesis supervised by Lara Bossinger at the Universidad Nacional Autónoma de México (UNAM). I am grateful for her advice and support during the realization of this work. I would also like to thank Juan González López for his help during the computational implementation.

\section{Definitions and notation}\label{sec1}
In this section, the general notation and definitions that will be used throughout this work will be introduced. Furthermore, a brief reminder on tropical geometry is presented in order to relate, in the next sections, the weighting matrices (obtained naturally from birational sequences) to weight vectors in the tropical Grassmannian.

Given a positive integer $n\in \Z^+$, define $\left[n\right]:=\left\lbrace 1,\dots,n\right\rbrace$. Given $k,n\in \Z^+$, it will be considered that $0<k<n$, unless otherwise stated. Let $I_{k,n}:=\left\lbrace \lbrace i_1<\dots<i_k\rbrace:(\forall l=1,\dots,k)(i_l\in [n])\right\rbrace$.

Given two vector spaces $V$ and $W$, write $V\leq W$ if $V$ is a subspace of $W$. For $k,n\in \Z^+$, the \emph{Grassmannian} of $k$-planes in $\C^n$ is defined as
\begin{equation*}
    \Gr\left(k,n\right) := \left\lbrace V\leq \C^n:\dim(V) = k \right\rbrace.
\end{equation*} 

The Grassmannian will be considered together with the Plücker embedding $\Gr\left(k,n\right)\hookrightarrow \P^{\binom{n}{k}-1}$. Let $L=(l_1,\dots,l_k)\in I_{k,n}$ and let $p_L$ be the Plücker variable defined for $V\in \Gr\left(k,n\right)$ as follows: given a basis $\lbrace v_i\rbrace_{i=1}^k$, let $M_V$ be the matrix with rows the vectors $v_i$ expressed in the standard basis of $\C^n$. Then 
\begin{equation*} 
p_L\left(M_V\right):=\det\left((M_V)_{l_1},\dots,(M_V)_{l_k}\right).
\end{equation*}
Let $I=(i_1<\dots i_{s-1}<i_{s+1}<\dots <i_k)\in I_{k-1,n}$ and $J\in (j_1<\dots <j_{s-1}<j<j_{s+1}<\dots j_{k+1})$. For $j\in I$, define $p_{I\cup j}:=0$. Otherwise, if $j\not\in I$, define
\begin{align*}
    I\cup j &:= \left(i_1<\dots < i_{s-1}<j<i_{s+1}<\dots< i_k\right).\\
    J\setminus j &:= \left(j_1<\dots<j_{s-1}<j_{s+1}<\dots <j_{k_1}\right). 
\end{align*} 

For $I\in I_{k-1,n}$ and $J\in I_{k+1,n}$, define the \emph{Plücker relation} $R_{I,J}$ as follows
\begin{equation}\tag{PR}\label{PR}
    R_{I,J} :=\sum_{j\in J}\left(-1\right)^{\#\lbrace i\in I:i<j\rbrace + \#\lbrace j'\in J:j<j' \rbrace}p_{I\cup j}p_{J\setminus j}.
\end{equation} The \emph{Plücker ideal} is the ideal generated by the Plücker relations
\begin{equation*}
    \I_{k,n} :=\left\langle R_{I,J}:I\in I_{k-1,n},J\in I_{k+1,n}\right\rangle\subset \C\left[p_K:K\in I_{k,n}\right],
\end{equation*} and the Grassmannian $\Gr\left(k,n\right)$ is the vanishing locus of this ideal. The homogeneous coordinate ring is $\C\left[\Gr(k,n)\right] = \C\left[p_K:K\in I_{k,n}\right]/\I_{k,n}$, and is denoted $A_{k,n}$. The images of the Plücker variables $p_I$ under the canonical projection $\C\left[p_K:K\in I_{k,n}\right]\rightarrow A_{k,n}$ are denoted by $\overline{p_I}$ and are called \emph{Plücker coordinates}.

\subsection*{Monomial orders, initial ideals, and tropicalization}
Given a polynomial ring $R=\C\left[x_1,\dots,x_n\right]$, $\Mon\left(R\right)$ denotes the set of \emph{monomials} of $R$, that is
\begin{equation*}
    \Mon\left(R\right) := \left\lbrace x_1^{a_1}\dots x_n^{a_n}: (\forall i=1,\dots,n)(a_i\in \N)\right\rbrace.
\end{equation*} For an element in $\Mon\left(R\right)$, write $\x^{\a}:= x_1^{a_1}\dots x_n^{a_n}$ with $\a\in \N^n$.

A total order $\leq$ on $\Mon\left(R\right)$ is called a \emph{monomial order} or \emph{term order} if it satisfies
\begin{itemize}
    \item $\forall \a\in \N^n$: $1\leq \x^a$;
    \item If $\x^{\a}\leq \x^{\b}$, then: $\forall \U\in \N^{n}$ it holds $\x^{\a+\U}\leq \x^{\b+\U}$.
\end{itemize}

Given a total order $\leq$ on $\Mon\left(R\right)$ and vector $\w\in \R^n$, called a \emph{weight vector}, a new order is defined as: $\x^{\a}\leq_{\mathbf{w}} \x^{\b}$ if and only if 
\begin{itemize}
    \item $\w\cdot\a < \w\cdot \b$, or
    \item $\w\cdot \a = \w\cdot \b$ and $\x^{\a}\leq \x^{\b}$.
\end{itemize}

Let $f=\sum_{\a\in \N^n}c_{\a}\x^{\a}\in R$ be an arbitrary element where only finitely many $c_{\a}\neq 0$, and $I\subset R$ an ideal. For the monomial order $\leq$, define the \emph{initial term} of $f$ and the \emph{initial ideal} of $I$ as
\begin{equation*}
    \In_{\leq} \left(f\right) := \min \left\lbrace c_{\a}\x^{\a}:c_{\a}\neq 0 \right\rbrace , \qquad \In_{\leq }\left(I\right) :=\left\langle \In_{\leq}(g):g\in I \right\rangle.
\end{equation*}

Let $w\in \R^n$ be a weight vector and $M\in \Z^{m\times n}$. Let $<$ and $\prec$ be the usual order on $\R$ and a monomial order on $\Z^m$, respectively. Define the \emph{initial form} of $f$ with respect to $w$ and $M$, respectively, as
\begin{equation*}
    \In_{w}\left(f\right) := \sum_{\a: \,w\cdot \a  = \min_{<}\lbrace w\cdot \b:c_{\b}\neq 0\rbrace}c_{\a}\x^{\a},\qquad \In_M \left(f\right) := \sum_{\a:\,M\a^t = \min_{\prec} \lbrace M\b^t: c_{\b}\neq 0\rbrace}c_{\a}\x^{\a},
\end{equation*} and in a similar fashion, define the \emph{initial ideals} as 
\begin{equation*}
    \In_w\left(I\right) := \left\langle \In_w(g):g\in I\right\rangle ,\qquad \In_M\left(I\right) :=\left\langle \In_M(g):g\in I\right\rangle.
\end{equation*}

By the definition of initial ideal with respect to a weight vector, there is a fan structure on $\R^n$, called the \emph{Gröbner fan} and denoted $GF\left(I\right)$. By \cite[Lemma 1.1]{sturmfels1996equationsdefiningtoricvarieties}, toric ideals are prime and \emph{binomial}, i.e., generated by binomials. It is, however, not always true that initial ideals with respect to arbitrary weight vectors $w\in \R^n$ are binomial. Thus one must consider a special subfan of $GF(I)$.

Consider now $\hat{R} = \C\left[x_1^{\pm 1},\dots , x_n^{\pm 1}\right]$, the ring of Laurent polynomials and $\hat{f}=\sum_{\a\in \Z^n}c_{\a}\x^{\a}\in \hat{R}$, with only finitely many $c_{\a}\neq 0$. The \emph{tropicalization} of $f$ (cf. \cite[Definitions 3.1.1/3.1.2]{maclagan2021introduction}) is defined  as the function $\hat{f}^{\trop}:\R^n\rightarrow \R$ given by 
\begin{equation*}
    \hat{f}^{\trop}\left(w\right) := \min \left\lbrace w\cdot \a:c_{\a}\neq 0\right\rbrace.
\end{equation*} Given the hypersurface $V(\hat{f})\subset \left(\C^*\right)^n$, the \emph{tropical hypersurface} $\trop\left(V(\hat{f})\right)$ is defined as
\begin{equation*}
    \trop\left(V(\hat{f})\right) :=\left\lbrace w\in \R^n:\quad \begin{matrix}
        \textrm{the minimum in }\hat{f}^{\trop}(w)\\
        \textrm{is achieved at least twice}
    \end{matrix}\right\rbrace,
\end{equation*} and for an ideal $I\subset \hat{R}$, the \emph{tropical variety} is defined as
\begin{equation*}
    \trop\left(V(I)\right) := \bigcap_{f\in I}\trop\left(V(f)\right).
\end{equation*}

If $I\subset R$, consider the ideal $\hat{I}:=I\hat{R}$. Then, the tropicalization of the variety $V\left(I\right)$ is defined as
\begin{equation*}
    \trop\left(V(I)\right) :=\trop\left(V(\hat{I})\right).
\end{equation*}

The following characterization of the tropicalization, regarded as the \emph{Fundamental Theorem of Tropical Geometry}, is stated as follows \cite[Theorem 3.2.3]{maclagan2021introduction}

\begin{thm*}
    Let $I\subset R$ be an ideal. Then
    \begin{equation*}
        \trop\left(V(I)\right) = \left\lbrace w\in \R^n: \In_w(I) \textrm{ is monomial free} \right\rbrace.
    \end{equation*}
\end{thm*}

By this characterization, there is a fan structure on $\trop\left(V(I)\right)$ defined by noticing that two vectors $w,u\in \trop\left(V(I)\right)$ belong to relative interior of the same cone if and only if $\In_u\left(I\right) = \In_w\left(I\right)$. Furthermore, this describes $\trop\left(V(I)\right)$ as a subfan of the Gröbner fan.

\section{Birational sequences}\label{sec2}
For $k,n\in \Z^+$, it is known that there is an identification of $\Gr\left(k,n\right)$ with $SL_n/P_k$, where $SL_n$ is the special linear complex group and $P_k$ is the parabolic subgroup of upper triangular block matrices of size $k\times k$ and $(n-k)\times (n-k)$ (cf. \cite[Chapter 5]{LakshmibaiBrown2015}). Consider the special linear Lie algebra $\sl_n$ and the exponential map $\exp:\sl_n\rightarrow SL_n$. Given the $\mathbf{A}_{n-1}$ type root system $\Phi$, let $\Phi^+=\lbrace \epsilon_i-\epsilon_j:(i<j)\in I_{2,n}\rbrace$ be the set of positive roots. For any $\beta=\epsilon_i-\epsilon_j\in \Phi^+$, define $f_{\beta}$ the $n\times n$ matrix with 1 in the entry $(i,j)$ and 0 in the others. Consider $x_{\beta}\in \C$ and $\exp\left(x_{\beta}f_{\beta}\right)\in SL_n$. Let
\begin{equation*}
    U_{\beta} :=\left\lbrace \mathbf{1}_{n\times n} + x_{\beta}f_{\beta}:x_{\beta}\in \C\right\rbrace\subset U^+,
\end{equation*} with $U^+\subset SL_n$ the subgroup of upper triangular matrices with 1s along the diagonal.

\begin{definition}[\cite{FANG2017107,Bossinger:2021}]\label{birseqs}
    Given $\beta_1,\dots,\beta_N\in \Phi^+$, the sequence $S=\left(\beta_1,\dots,\beta_N\right)$ is \emph{birational} for $\Gr\left(k,n\right)$ if the multiplication map
    \begin{equation*}
        \psi_S: U_{\beta_N}\times \dots \times U_{\beta_1}\rightarrow U^+
    \end{equation*} has image birational to $\Gr\left(k,n\right)$.
\end{definition}

Consider the following example.
\begin{ex}[PBW sequence]
    Let $\Phi_k^+ := \left\lbrace \epsilon_i-\epsilon_j:1\leq i\leq k<j\leq n\right\rbrace$ and let $S=\left(\beta_1,\dots,\beta_N\right)$ be any ordering of the elements of $\Phi_k^+$. This is called a \emph{PBW sequence}. The image of the multiplication map has elements of the form
    \begin{equation*}
        A = \begin{pmatrix}
            \mathbf{1}_{k\times k} & * \\
            \mathbf{0} & \mathbf{1}_{(n-k)\times (n-k)}
        \end{pmatrix},
    \end{equation*}  and the span of the first $k$ rows gives the required birational map. It is a straightforward computation to show that the ordering of the roots in this example does not modify the birational map.
\end{ex}

It will be seen that certain PBW sequences can be considered as belonging to the special subclass of \emph{iterated sequences}. In order to define this class, the following Lemma (cf. \cite[Lemma 1]{Bossinger:2021}) is needed.

\begin{lem}\label{iterated}
    Let $\beta_1,\dots,\beta_N\in \Phi^+$ (in $\sl_n$) and let $S' = \left(\beta_1,\dots,\beta_N\right)$ be a birational sequence for $\Gr\left(k,n\right)$. Let $i_1,\dots,i_k\in [n]$ be pairwise different indices. Then 
    \begin{equation*}
        S = \left(\epsilon_{i_1}-\epsilon_{n+1},\dots,\epsilon_{i_k}-\epsilon_{n+1},\beta_1,\dots,\beta_N\right)
    \end{equation*} is a birational sequence for $\Gr\left(k,n+1\right)$.
\end{lem}

This lemma helps construct new birational sequences which are not necessarily PBW. Furthermore, the proof gives a detailed construction of the birational map, describing both the map and its pullback. It is now possible to define the special class of \emph{iterated sequences}.

\begin{definition}\label{itseqs}
    Let $S'$ be a birational sequence for $\Gr\left(k,k+1\right)$. Obtain a birational sequence $S$ for $\Gr\left(k,n\right)$ by applying \cite[Lemma 1]{Bossinger:2021} $n-k-1$ times to $S'$. Then $S$ is called an \emph{iterated sequence}.
\end{definition}

\subsection*{Lowest-term valuations}
Now that the birational sequences have been defined and the examples of PBW and iterated sequences have been presented, it is possible to define a valuation on $A_{k,n}\setminus \lbrace 0\rbrace$.

Fix a birational sequence $S= \left(\beta_1,\dots,\beta_N\right)$ for $\Gr\left(k,n\right)$ and denote $\phi_S:\im\left(\psi_S\right)\dashrightarrow \Gr\left(k,n\right)$ the birational map. Consider the polynomial ring $\C\left[x_{\beta}:\beta\in S\right]$. Fix an identification
\begin{equation}\tag{*}
    x_{\beta_1}^{a_1}\dots x_{\beta_N}^{a_N} \quad \longleftrightarrow \quad \left(a_1,\dots,a_N\right). 
\end{equation} and let $<_{lex}$ be the lexicographic order on $\Z^N$.

Recall that a valuation is a function $\V_S:\C\left[x_{\beta}:\beta\in S\right]\setminus \lbrace 0\rbrace \rightarrow \left(\Z^N,\prec_S\right)$, with $\prec_S$ a monomial order on $\Z^N$, that satisfies: for all $f,g\in \C\left[x_{\beta}:\beta\in S\right]\setminus \lbrace 0\rbrace$ and $r\in \C^*$ it holds
\begin{equation*}
    \min_{\prec_S}\lbrace \V_S(f),\V_S(g)\rbrace\preceq_S \V_S\left(f+g\right),\qquad \V_S\left(fg\right) = \V_S\left(f\right) + \V_S\left(g\right), \qquad \V_S\left(rf\right) = \V_S\left(f\right).
\end{equation*}

In order to define the function $\V_S$, start by defining the \emph{height function} $\Ht:R^+\rightarrow \Z^+$ and the \emph{height-weighted function} $\Psi_S:\Z^N\rightarrow \Z^+$ as
\begin{equation*}
    \Ht\left(\epsilon_i-\epsilon_j\right) := j-i,\qquad \Psi_S\left(m_1,\dots,m_N\right) := \sum_{l=1}^N m_l\Ht(\beta_l).
\end{equation*}

The $\Psi_S$\emph{-weighted reverse lexicographic order} $\preceq_{S}$ is the monomial order defined as follows: $\x^{\a}\preceq_{\Psi_S} \x^{\b}$ if and only if $\Psi_S\left(\a\right)< \Psi_S\left(\b\right)$ or ($\Psi_S\left(\a\right) = \Psi_S\left(\b\right)$ and  $\a \geq_{lex} \b$).

Let $f=\sum_{\a\in \N^N}c_{\a}\x^{\a}$, with only finitely many $c_{\a}\neq 0$. Consider the function  defined by
\begin{equation*}
    \V_S:\C\left[x_{\beta}:\beta\in S\right]\setminus\lbrace 0\rbrace \rightarrow \left(\Z^N,\prec_S\right),\qquad \V_S\left(f\right) := \min_{\prec_S}\left\lbrace \a\in \N^N:c_{\a}\neq 0 \right\rbrace.
\end{equation*}
This function is, in fact, a valuation on $\C\left[x_{\beta}:\beta\in S\right]\setminus \lbrace 0\rbrace$. It is extended to $\C\left(x_{\beta}:\beta\in S\right)\setminus \lbrace 0\rbrace$ as $\V_S\left(\frac{f}{g}\right):=\V_S\left(f\right)-\V_S\left(g\right)$. Using the birational map $\phi_S$, define it on $\C\left(\Gr(k,n)\right)\setminus \lbrace 0\rbrace$ by $\V_S\left(h\right) :=\V_S\left(\phi_S^*(h)\right)$, and restrict to $A_{k,n}\setminus \lbrace 0\rbrace$. This is called the \emph{lowest term valuation} associated to the biration sequence $S$.

\subsection*{Weight homogeneity}
Given the definition of the lowest term valuation associated to a birational sequence $S$ for $\Gr\left(k,n\right)$, the next step is computing the valuations of Plücker coordinates in order to get the weighting matrix $M_{\V_S}$. This might, at first, seem not so straightforward since first the values of $\Psi_S$ must be computed and compared, and in the case of a tie, procede with the lexicographic order. 

It turns out that if $k=3$ and $S$ is iterated, this computation requires only the comparison using the lexicographic order. In order to prove this, the following lemma is necessary.

\begin{lem}[Weight Homogeinity]\label{homogeneidad}
    Let $S$ be an iterated sequence for $\Gr\left(3,n\right)$, $I=(i_1,i_2,i_3)\in I_{3,n}$ and $\overline{p_I}\in A_{3,n}$. Let $\a\in \N^{3(n-3)}$ be such that $\overline{\xa}\in \supp\left(\phi_S^*(\overline{p_I})\right)$. Then
    \begin{equation}
        \Psi_S\left(\a\right)=i_1+i_2+i_3-6.
    \end{equation}
\end{lem}
\begin{proof}
    Without loss of generality, let $(\beta_1,\beta_2,\beta_3)$ be a PBW sequence and $S$ the iterated sequence
    \begin{equation*}
        S=\left(\epsilon_{l_1}-\epsilon_n,\epsilon_{l_2}-\epsilon_n, \epsilon_{l_3}-\epsilon_n,\dots, \beta_1,\beta_2,\beta_3\right).
    \end{equation*} 

    Procede by induction: let $I=(i_1,i_2,i_3)\in I_{3,4}\setminus I_{3,3}$ and $j\in [3]$ be such that $I=\left((1,2,3)\setminus j\right)\cup 4$. Then $\supp\left( \phi_S^*\overline{p_I})\right)=\left\lbrace x_{j,4}\right\rbrace=\left\lbrace \x^{\a_{j,4}}\right\rbrace$ and
    \begin{equation*}
        \Psi_S\left(\a_{j,4}\right) = 4 - j = 1 + 2 + 3 + 4 -j -6 = i_1 + i_2 + i_3 - 6.
    \end{equation*}

    Assume that for all $J=(j_1,j_2,j_3)\in I_{k,n-1}$ and all $\a\in \Z^{3(n-3)}$ with $\xa\in \supp(\phi_S^*(\overline{p_I}))$, the following equality holds
    \begin{equation*}
        \Psi_S(\a)= j_1+j_2+j_3 - 6.
    \end{equation*}

    Let $I=(i_1,i_2,i_3)\in I_{3,n}\setminus I_{3,n-1}$. By the proof of \cite[Lemma 1]{Bossinger:2021}, the following equality holds
    \begin{equation*}
        \supp\left(\phi_S^*(\overline{p_I})\right)=\left\lbrace x_{l,n}\x^{\b}: \x^{\b}\in \supp(\phi_S^*(\overline{p_{I\setminus n\cup l}})),l\in \lbrace l_1,l_2,l_3\rbrace \wedge (l\not \in I\setminus n)\right\rbrace.
    \end{equation*}

    Let $l'\in (l_1,l_2,l_3)$ be such that $x_{l',n}\x^{\b}\in \supp(\phi_S^*(\overline{p_I}))$, with $\x^{\b}\in \supp\left(\phi_S^*(p_{I\setminus n\cup l'})\right)$. By induction, it holds
    \begin{equation*}
        \Psi(\b)= i_1 + i_2+i_3-n+l' -6. 
    \end{equation*} Let $\a\in \Z^{3(n-3)}$ be such that $\xa = x_{l',n}$. Notice that $\x^{\a+\b}=x_{l',n}\xb$. Then
    \begin{equation*}
        \Psi(\a+\b)=n-l' + i_1+i_2+i_3-n+l'-6 = i_1 + i_2+i_3-6.
    \end{equation*}
\end{proof}

The direct consequence of this lemma is Corollary \ref{only-depends-on-lex}, which significantly reduces the complexity of computations via the Algorithm \ref{algorithm}.

\begin{cor}\label{only-depends-on-lex}
    Let $S$ be an iterated sequence for $\Gr\left(3,n\right)$, $I=(i_1,i_2,i_3)\in I_{3,n}$ and $\overline{p_I}\in A_{3,n}$. Then
    \begin{equation*}
        \V_S\left(\overline{p_I}\right) = \max_{<_{lex}} \left\lbrace \mathbf{m}\in \Z^d:\x^{\mathbf{m}}\in \supp (\phi_S^*(\overline{p_I}))\right\rbrace.
    \end{equation*}
\end{cor}

In order to describe Algorithm \ref{algorithm}, used to compute the lowest term valuations, the following notation will be used: if $\lbrace i_1,i_2,i_3 \rbrace\subset [r-1]$ is the set used to iterate from a sequence for $\Gr\left(k,r-1\right)$ to a sequence for $\Gr\left(k,r\right)$, then $\beta_{i_j,r} := \epsilon_{i_j}-\epsilon_r$ for $j=1,2,3$. Furthermore, $f_l\in \Z^{3(n-3)}$ will denote the unitary vector with 1 in the $l$-th entry and 0 in the others.

\begin{algorithm}[H]
\caption{Computation of \emph{lowest term valuation} $\V_S$ of a Plücker coordinate}
\begin{algorithmic}[1]
\Require An iterated sequence $S$ and a multiindex $I\in I_{3,n}$.
\Statex
\State \textbf{Start:} $r=0$, $\mathbf{m}=\mathbf{0}\in \Z^{3(n-3)}$.
\While{$n-r\geq 4$}
    \State Compute $\V_S\left(\overline{p_I}\right)$ by adding unitary vectors and modifying the multiindex $I$.
    \If{$n-r\not\in I$}
        \State $r = r+1$.
    \Else
        \While{$n-r\in I$}
        \State $j=1$ and $\beta_{i_j,n-r}$
        \If{$i_j\not\in I$}
            \State $I= I\setminus (n-r)\cup i_{j}^{n-r}$, $\mathbf{m} = \mathbf{m}+f_{r+j}$, $r=r+1$.
        \Else 
        \State $j = j+1$.
        \EndIf
        \EndWhile
    \EndIf
\EndWhile
\Ensure $\mathbf{m} = \V_S(\overline{p_I})$.
\end{algorithmic}
\label{algorithm}
\end{algorithm}

Besides its usefulness in the computational implementation, Algorithm \ref{algorithm} also implies the following fact: the lowest term valuation $\V_S:\C\left[\Gr(3,n)\right]\setminus \lbrace 0\rbrace \rightarrow \Z^{3(n-3)}$ is a full-rank valuation (cf. Proof of \cite[Theorem 1]{Bossinger:2021}). To prove this it is enough to find a submatrix of the weighting matrix which is triangular with 1s along the diagonal. For example, when considering an iterated PBW sequence $S$ for the Grassmannian $\Gr\left(3,6\right)$ written as
\begin{equation*}
    S = \left(\beta_{1,6},\beta_{2,6},\beta_{3,6},\beta_{1,5},\beta_{2,5},\beta_{3,5},\beta_{1,4},\beta_{2,4},\beta_{3,4}\right).
\end{equation*}
Computing the lowest term valuation of the following Plücker coordinates and arranging them, in the exact order, in the weighting matrix, gives the matrix with 1s along the diagonal:
\begin{equation*}
    \overline{p_{(4,5,6)}} , \overline{p_{(1,5,6)}}, \overline{p_{(1,2,6)}}, \overline{p_{(3,4,5)}}, \overline{p_{(1,4,5)}}, \overline{p_{(1,2,5)}}, \overline{p_{(2,3,4)}}, \overline{p_{(1,3,4)}}, \overline{p_{(1,2,4)}}.
\end{equation*}

\section{Initial ideals}\label{sec3}
Let $S$ be a fixed iterated sequence for $\Gr\left(3,n\right)$, and assume without losing generality that it is iterated from a PBW sequence. By \cite[Theorem 1]{Bossinger:2021} and the Fundamental Theorem of Tropical Geometry, it is known that the initial ideal $\In_{M_S}\left(\I_{3,n}\right)$ is monomial-free. Since the set of Plücker relations is a (subset of a) Gröbner basis of $\I_{3,n}$ with respect to the weigthing matrix $M_S$, a first step in proving that the initial ideal is not only monomial-free but also binomial, is proving that the initial form of any Plücker relation is binomial.

Let $\V_S$ be the lowest term valuation, $d = 3(n-3)$ and $N=\binom{n}{3}$. Let $M_S\in M_{d\times N}\left(\Z\right)$ be the weigthing matrix $M_S := M_{\V_S}$ given by
\begin{equation}\label{weight-matrix}
    M_S := \begin{pmatrix}
        \V_S(\overline{p_K})^t
    \end{pmatrix}_{K\in I_{3,n}}.
\end{equation}

Let $I\in I_{2,6}, J\in I_{4,6}$ and $R_{I,J}\neq 0$ a Plücker relation. The number of terms of $R_{I,J}$ is
\begin{itemize}
    \item 3, if $I$ and $J$ have a common index;
    \item 4, if $I$ and $J$ have no common index.
\end{itemize}

For all $j\in J$ such that $p_{I\cup j}p_{J\setminus j}\neq 0$, let $\mathbf{m}_j\in \Z^{d}$ be the vector that satisfies $\mathbf{p}^{\mathbf{m}_j}=p_{I\cup j}p_{J\setminus j}$. From the definition of $M_S$, it holds
\begin{equation*}
    M_S\mathbf{m}_j^t = \V_S\left(\overline{p_{I\cup j}}\right) + \V_S\left(\overline{p_{J\setminus j}}\right) = \V_S\left(\overline{p_{I\cup j}p_{J\setminus j}}\right). 
\end{equation*}

The following remark is a consequence of Lemma \ref{homogeneidad}.
\begin{rem}\label{min-max}
    Let $S$ be an iterated sequence $\Gr\left(3,n\right)$. Let $I=(i_1,i_2)\in I_{2,n}, J=(j_1,j_2,j_3,j_4)\in I_{4,n}$ and $R_{I,J}\neq 0$. For all $\mathbf{m}\in \Z^{N}$ such that $\mathbf{p}^{\mathbf{m}}\in \supp(R_{I,J})$, the following equality holds
    \begin{equation*}
        \Psi_S\left(\mathbf{m}\right) = i_1+i_2+j_1+j_2+j_3+j_4 - 12.
    \end{equation*} Therefore
    \begin{equation*}
        \min_{\prec_{\Psi_S}}\left\lbrace M_S\mathbf{m}^t:\p^{\mathbf{m}}\in \supp(R_{I,J})\right\rbrace = \max_{<_{lex}}\left\lbrace  M_S\mathbf{m}^t:\p^{\mathbf{m}}\in \supp(R_{I,J})\right\rbrace.
    \end{equation*}
\end{rem}

In turn, this remark has the following consequence.

\begin{prop}\label{binomial}
    Let $\left(\beta_1,\beta_2,\beta_3\right)$ be a PBW sequence for $\Gr\left(3,4\right)$ and $S$ an iterated sequence for $\Gr\left(3,n\right)$. Let $I\in I_{2,n}, J\in I_{4,n}$ and $R_{I,J}\neq 0$. Then $\In_{M_S}\left(R_{I,J}\right)$ is binomial.
\end{prop}

By the Remark \ref{min-max}, finding the minima with respect to $\prec_{\Psi_S}$ is equivalent to finding the maxima with respect to $<_{lex}$. This fact will be used throughout the proof of Proposition \ref{binomial}, as well as the Algorithm \ref{algorithm} to compute the valuations. The notation used for the Algorithm \ref{algorithm} will be used here as well: for $u\in [3]$,  $f_u\in \Z^3$ is the unitary vector with $1$ in the $u$-th entry.

The proof is now divided in the two cases, corresponding to the number of terms of a Plücker relation.

\begin{proof}[Proof in the case of 3 terms]
    The proof is done by induction, starting with a Plücker relation for $\Gr\left(3,5\right)$. Consider the following notation for the birational sequences:
    \begin{align*}
        \Gr(3,4): \qquad & S_B  = \left(\beta_1,\beta_2,\beta_3\right),\\
        \Gr(3,5): \qquad & S''  = \left(\epsilon_{l_1}-\epsilon_5,\epsilon_{l_2}-\epsilon_5,\epsilon_{l_3}-\epsilon_5,S_B\right), \\
        \Gr(3,n-1): \qquad &S'  = \left(\epsilon_{j_1}-\epsilon_{n-1},\epsilon_{j_2}-\epsilon_{n-1},\epsilon_{j_3}-\epsilon_{n-1}\dots,S''\right),\\
        \Gr(3,n): \qquad  &S  = \left(\epsilon_{h_1}-\epsilon_n,\epsilon_{h_2}-\epsilon_n,\epsilon_{h_3}-\epsilon_n,S'\right).
    \end{align*} 

    \begin{proof}[Base case]
    There are two different cases
        \begin{itemize}
            \item Let $I=(r,5),J=(s_1,s_2,s_3,r)$, with $r,s_1,s_2,s_3\in [4] $ pairwise different.

            Let $t=\min\lbrace t'\in [3]:l_{t'}\neq r\rbrace$. Since $\lbrace r,s_1,s_2,s_3\rbrace=[4]$, then $l_t\in (s_1,s_2,s_3)$. Write $(s_1',s_2')=(s_1,s_2,s_3)\setminus l_t$ and compute the valuations to get
            \begin{align*}
                \V_S\left(\overline{p_{I\cup s_1'}p_{J\setminus s_1'}}\right) &= \left(f_t,\V_{S_B}(\overline{p_{I\cup s_1'\setminus 5\cup l_t}p_{J\setminus s_1'}})\right),\\
                \V_S\left(\overline{p_{I\cup s_2'}p_{J\setminus s_2'}}\right) &= \left(f_t,\V_{S_B}(\overline{p_{I\cup s_2'\setminus 5\cup l_t}p_{J\setminus s_2'}})\right),\\
                \V_S\left(\overline{p_{I\cup l_{t}}p_{J\setminus l_t}}\right) &= \left(f_u,\V_{S_B}(\overline{p_{I\cup l_t\setminus 5\cup l_{u}}p_{J\setminus l_t}})\right),
            \end{align*} with $u\in [3]$ satisfying $u>t$. Furthermore, the following equation holds
            \begin{equation*}
                \overline{p_{(r,5)\cup s_1' \setminus 5 \cup l_t}p_{(s_1,s_2,s_3,r)\setminus s_1'}} = \overline{p_{(r,5)\cup s_2' \setminus 5 \cup l_t}p_{(s_1,s_2,s_3,r)\setminus s_2'}}.
            \end{equation*} Therefore
            \begin{equation*}
                \V_S\left(\overline{p_{I\cup s_1'}p_{J\setminus s_1'}}\right) = \V_S\left(\overline{p_{I\cup s_2'}p_{J\setminus s_2'}}\right)>_{lex}\V_S\left(\overline{p_{I\cup l_{t}}p_{J\setminus l_t}}\right).
            \end{equation*}

            \item Let $I=(4,5)$ and $J=(1,2,3,5)$. Since $l_1,l_2,l_3\in [4]$, there is a term in $R_{I,J}$ of the form $p_{(l_1,l_2,5)}p_{K}$, with $K=(4,5)\cup s$ o $K=(1,2,3,5)\setminus s$ for some $s\in \lbrace 1,2,3\rbrace$. Let $K_1,K_2,L_1,L_2$ be such that
            \begin{equation*}
                p_{(l_1,l_2,5)}p_K,p_{L_1}p_{K_1}, p_{L_2}p_{K_2}\in \supp\left(R_{I,J}\right),
            \end{equation*} and the valuations are 
            \begin{align*}
                \V_S\left(\overline{p_{(l_1,l_2,5)}p_K}\right) &= \left(f_1+f_3,\mathbf{n}'\right)\\
                \V_S\left(\overline{p_{L_1}p_{K_1}}\right) &= \left(f_1+f_2,\mathbf{n}_1\right)\\
                \V_S\left(\overline{p_{L_2}p_{K_2}}\right) &= \left(f_1+f_2,\mathbf{n}_2\right),
            \end{align*} with $\mathbf{n}',\mathbf{n}_1,\mathbf{n}_2\in \Z^{3(n-1-3)}$. Therefore
            \begin{equation*}
                \V_S\left(\overline{p_{L_1}p_{K_1}}\right) = \V_S\left(\overline{p_{L_2}p_{K_2}}\right) >_{lex} \V_S\left(\overline{p_{(l_1,l_2,5)}p_{K_1}}\right).
            \end{equation*}
        \end{itemize}
        This proves the base case.
    \end{proof}

    Assume that for all Plücker relation $R_{I,J}$ for $\Gr\left(k,n-1\right)$, the initial form $\In_{M_S}\left(R_{I,J}\right)$ is binomial, this is, there are $s_1,s_2,s_3\in J$ such that
    \begin{equation*}
        \V_{S'}\left(\overline{p_{I\cup s_1}p_{J\setminus s_1}}\right) = \V_{S'}\left(\overline{p_{I\cup s_2}p_{J\setminus s_2}}\right) >_{lex} \V_{S'}\left(\overline{p_{I\cup s_3}p_{J\setminus s_3}}\right).
    \end{equation*}

    Let $I\in I_{2,n},J\in I_{4,n}$, and $R_{I,J}\neq 0$ a Plücker relation such that $n\in I\setminus J$ or $n\in I\cap J$. For the inductive step, notice that for $n-1\geq 5$, the equality $\lbrace s_1,s_2,s_3,r\rbrace=[n-1]$ does not hold. Consider the two previous cases.
    \begin{itemize}
        \item Let $I=(r,n)$ and $J=(s_1,s_2,s_3,r)$, with $s_1,s_2,s_3,r\neq n$ pairwise different.
        
        Let $t=\min \lbrace t'\in [3]:h_{t'}\neq r\rbrace$. If $h_t\in (s_1,s_2,s_3)$, the same argument used in the base case applies by changing $5$ for $n$. If this is not the case, for $s\in (s_1,s_2,s_3)$, the valuations are
        \begin{align*}
            \V_S\left(\overline{p_{I\cup s}p_{J\setminus s}}\right) & = \left(f_t,\V_{S'}(\overline{p_{I\cup s\setminus n\cup h_t}p_{J\setminus s}})\right).
        \end{align*} The following equality holds
        \begin{equation*}
            \left\lbrace p_{I\cup s\setminus n\cup h_t}p_{J\setminus s} : s=s_1,s_2,s_3\right\rbrace = \supp\left(R_{I\setminus n\cup h_t,J}\right).
        \end{equation*} By induction, $\In_{M_S}\left(R_{I\setminus n\cup h_t,J}\right)$ is binomial. Therefore, $\In_{M_S}\left(R_{I,J}\right)$ is binomial.

        \item Let $I=(r,n)$ and $J=(s_1,s_2,s_3,n)$, with $s_1,s_2,s_3,r\neq n$ pairwise different.

        If $\lbrace h_1,h_2\rbrace\subset \lbrace s_1,s_2,s_3,r \rbrace$, the same argument used in the base case applies by changing $5$ for $n$. Otherwise, use the inductive hypothesis to ``pass to a lower dimension". This argument is exemplified in the case $h_1\not \in \lbrace s_1,s_2,s_3,r\rbrace$: for $s=s_1,s_2,s_3$, the valuations are
        \begin{equation*}
            \V_S\left(2f_1,\V_{S'}(\overline{p_{I\cup s\setminus n\cup h_1}p_{J\setminus s\setminus n\cup h_1}})\right),
        \end{equation*} and the following equality holds
        \begin{equation*}
            \left\lbrace p_{I\cup s\setminus n\cup h_1}p_{J\setminus s\setminus n\cup h_1}:s=s_1,s_2,s_3\right\rbrace = \supp\left(R_{I\setminus n\cup h_1,J\setminus n\cup h_1}\right).
        \end{equation*} By induction, $\In_{M_S}\left(R_{I\setminus n\cup h_1,J\setminus n\cup h_1}\right)$ is binomial. Therefore, $\In_{M_S}\left(R_{I,J}\right)$ is binomial.
    \end{itemize}
\end{proof}

\begin{proof}[Proof in the case of 4 terms]
    The proof is done by induction, starting with a Plücker relation for $\Gr\left(3,6\right)$. Consider the notation $S_B, S'',S',S$ of the previous proof, and add
    \begin{equation*}
        S''' = \left(\epsilon_{b_1}-\epsilon_6,\epsilon_{b_2}-\epsilon_6,\epsilon_{b_3}-\epsilon_6,S''\right),
    \end{equation*} for the iterated sequence for $\Gr\left(3,6\right)$.

    \begin{proof}[Base case]
        Let $I\in I_{2,6}$, $J\in I_{4,6}$, and $R_{I,J}\neq 0$ a Plücker relation with four terms for $\Gr\left(3,6\right)$. Consider the following two cases
        \begin{itemize}
            \item Let $I=(r,6)$ and $J=(s_1,s_2,s_3,s_4)$ with $\lbrace r,s_1,s_2,s_3,s_4\rbrace=[5]$.

            Let $t=\min\lbrace t'\in [3]:b_{t'}\neq r\rbrace$. Then $b_t\in \lbrace s_1,s_2,s_3,s_4\rbrace$. For all  $s'\in \left(s_1,s_2,s_3,s_4\right)\setminus b_t$ the following inequality holds
            \begin{equation*}
                \V_{S'''}\left(\overline{p_{I\cup s'}p_{J\setminus s'}}\right) >_{lex} \V_{S'''}\left(\overline{p_{I\cup b_t}p_{J\setminus b_t}}\right)
            \end{equation*} and
            \begin{equation*}
                \V_{S'''}\left(\overline{p_{I\cup s'}p_{J\setminus s'}}\right) = \left(f_t,\V_{S''}(\overline{p_{I\cup s'\setminus 6\cup b_t}p_{J\setminus s'}})\right).
            \end{equation*} The following equality holds 
            \begin{equation*}
                \left\lbrace p_{I\cup s'\setminus 6\cup b_t}p_{J\setminus s'} : s'\in (s_1,s_2,s_3,s_4)\setminus b_t \right\rbrace = \supp\left(R_{I\setminus 6\cup b_t,J}\right),
            \end{equation*} and by the previous proof, $\In_{M_S}\left(R_{I\setminus 6\cup b_t,J}\right)$ is binomial. Therefore, $\In_{M_S}\left(R_{I,J}\right)$ is binomial.

            \item Let $I=(r_1,r_2)$ and $J=(s_1,s_2,s_3,6)$, with $\lbrace r_1,r_2,s_1,s_2,s_3\rbrace=[5]$.

            Since $b_1\in [5]$, there are two cases: $b_1\in (r_1,r_2)$ or $b_1\in (s_1,s_2,s_3)$. 

            In the first case, for all $s\in (s_1,s_2,s_3)$, it holds
            \begin{equation*}
                \V_{S'''}\left(\overline{p_{I\cup s}p_{J\setminus s}}\right)>_{lex} \V_{S'''}\left(\overline{p_{I\cup 6}p_{J\setminus 6}}\right)
            \end{equation*} and the following equality holds
            \begin{equation*}
                \left\lbrace p_{I\cup s}p_{J\setminus s\setminus 6\cup b_1}:s\in (s_1,s_2,s_3) \right\rbrace = \supp\left(R_{I,J\setminus 6\cup b_1}\right).
            \end{equation*} By the previous proof, $\In_{M_S}\left(R_{I,J\setminus 6\cup b_1}\right)$ is binomial. Therefore, $\In_{M_S}\left(R_{I,J}\right)$ is binomial.

            For the second case, for $s =  b_1,6$, $r\in (s_1,s_2,s_3)\setminus b_1$, it holds
            \begin{equation*}
                \V_{S'''}\left(\overline{p_{I\cup s}p_{J\setminus s}}\right) >_{lex} \V_{S'''}\left(\overline{p_{I\cup r}p_{J\setminus r}}\right).
            \end{equation*} Writing explicitly the left hand side for $s=6,b_1$, the valuations are
            \begin{align*}
                \V_{S'''}\left(\overline{p_{I\cup s}p_{J\setminus s}}\right) &= \V_{S'''}\left(\overline{p_{(r_1,r_2)\cup 6}p_{(s_1,s_2,s_3,6)\setminus 6}}\right) = \left(f_1,\V_{S''}(\overline{p_{(r_1,r_2)\cup b_1}p_{(s_1,s_2,s_3)}})\right)\\
                 \V_{S'''}\left(\overline{p_{I\cup s}p_{J\setminus s}}\right) &= \V_{S'''}\left(\overline{p_{(r_1,r_2)\cup b_1}p_{(s_1,s_2,s_3,6)\setminus b_1}}\right) = \left(f_1,\V_{S''}(\overline{p_{(r_1,r_2)\cup b_1}p_{(s_1,s_2,s_3,6)\setminus b_1\setminus 6\cup b_1}})\right),
            \end{align*} and thus, $\In_{M_S}\left(R_{I,J}\right)$ is binomial.
        \end{itemize}
    \end{proof}

    Assume that for all Plücker relations $R_{I,J}\neq 0$ with four terms for $\Gr\left(3,n-1\right)$, its initial form $\In_{M_S}\left(R_{I,J}\right)$ is binomial, and consider the two previous cases.

    \begin{itemize}
        \item Let $I = \left(r,n\right)$ and $J = \left(s_1,s_2,s_3,s_4\right)$ with $r,s_1,s_2,s_3,s_4\in [n-1]$ pairwise different.

        Let $t=\min \lbrace t':h_{t'}\neq r \rbrace$. If $h_t\in (s_1,s_2,s_3,s_4)$, then the same argument used in the base case applies by changing 5 for $n$. Otherwise, if  $h_t\not\in (s_1,s_2,s_3,s_4)$, then the valuations are 
        \begin{equation*}
            \V_S\left(\overline{p_{I\cup s}p_{J\setminus s}}\right) = \left(f_t,\V_{S'}(\overline{p_{I\cup s\setminus n\cup h_t}p_{J\setminus s}})\right),
        \end{equation*} and the following equality holds
        \begin{equation*}
            \left\lbrace p_{I\cup s \setminus n\cup h_t}p_{J\setminus s}:s\in (s_1,s_2,s_3,s_4) \right\rbrace = \supp\left(R_{I\setminus n\cup h_t,J}\right),
        \end{equation*} with $R_{I\setminus n\cup h_t,J}$ a Plücker relation with 4 terms for $\Gr\left(k,n-1\right)$. By induction, $\In_{M_S}\left(R_{I\setminus n\cup h_t,J}\right)$ is binomial. Therefore, $\In_{M_S}\left(R_{I,J}\right)$ is binomial.

        \item Let $I=(r_1,r_2)$ y $J=(s_1,s_2,s_3,n)$, with $r,s_1,s_2,s_3,s_4\in [n-1]$ pairwise different. 

        If $h_1\in (r_1,r_2,s_1,s_2,s_3)$, the same argument used in the base case applies. If this is not the case, this is, if $h_1\not\in (r_1,r_2,s_1,s_2,s_3)$, then for all $s\in (s_1,s_2,s_3)$, the valuations are 
        \begin{align*}
            \V_S\left(\overline{p_{I\cup n}p_{J\setminus n}}\right) &= \left(f_1,\V_{S'}(\overline{p_{I\cup h_1}p_{J\setminus n}})\right),\\
            \V_S\left(\overline{p_{I\cup s}p_{J\setminus s}}\right) &=\left(f_1,\V_{S'}(\overline{p_{I\cup s}p_{J\setminus s\setminus n\cup h_1}})\right).
        \end{align*}  The following equality holds
        \begin{equation*}
            \left\lbrace p_{I\cup h_1}p_{J\setminus n},p_{I\cup s}p_{J\setminus s\cup n\cup h_1}:s=s_1,s_2,s_3\right\rbrace = \supp\left(R_{I,J\setminus n\cup h_1}\right).
        \end{equation*} By induction, $\In_{M_S}\left(R_{I,J\setminus n\cup h_1}\right)$ is binomial. Therefore, $\In_{M_S}\left(R_{I,J}\right)$ is binomial.
    \end{itemize}
\end{proof}

\subsection*{Initial ideals}
Denote $\mathcal{S}_{3,n}:=\left\lbrace \textrm{iterated sequences for }\Gr(3,6)\right\rbrace$. Let $S\in \mathcal{S}_{3,n}$. By \cite[Theorem 1]{Bossinger:2021}, there exists $C_S\subset \trop\left(\Gr(3,n)\right)$ such that 
\begin{equation}\label{igualdad-de-iniciales}
    \In_{M_S}\left(\I_{3,n}\right) = \In_{C_S}\left(\I_{3,n}\right).
\end{equation}

In order to find $C_S$, the notion of \emph{order preserving projection} is needed: given the weighting matrix $M_S:=\left(\V_S(\overline{p_K})^t\right)_{K\in I_{3,n}}$, a projection $e_S:\Z^{3(n-3)}\rightarrow \Z$ is \emph{order preserving} if
\begin{equation*}
    \In_{M_S}\left(\I_{3,n}\right) = \In_{e_S M_S}\left(\I_{3,n}\right),
\end{equation*} with $e_S M_S := \left(e_S(\V_S(\overline{p_K}))\right)_{K\in I_{3,n}}$. If this is the case, then $w_S := e_SM_S\in \trop\left(\Gr(3,n)\right)$, which in turn defines $C_S$.

The inequalities from the computation of initial forms of Plücker relations will be used to compute these projections. Let $R_{I,J}\neq 0$ ($I\in I_{2,n},J\in I_{4,n}$) be a Plücker relation. A projection $e_S:\Z^{3(n-3)}\rightarrow \Z$ \emph{preserves inequalities} with respect to $S$ if the following condition is satisfied: if $s_1,s_2\in J$ satisfy 
\begin{equation*}
    \V_S\left(\overline{p_{I\cup s_1}p_{J\setminus s_1}}\right) \prec_{\Psi_S} \V_S\left(\overline{p_{I\cup s_2}p_{J\setminus s_2}}\right),
\end{equation*} then
\begin{equation*}
    e_S\left(\V_S\overline{p_{I\cup s_1}p_{J\setminus s_1}})\right) <e_S\left(\V_S(\overline{p_{I\cup s_2}p_{J\setminus s_2}})\right).
\end{equation*}

By the proof of \cite[Lemma 3]{Bossinger_2020}, before concluding that Equation \eqref{igualdad-de-iniciales} holds for the vector obtained only from the inequalities of the initial forms of Plücker relations, it is necessary to verify that the set $\left\lbrace R_{I,J}\neq 0:I\in I_{2,n},J\in I_{4,n}\right\rbrace$ is a reduced Gröbner basis for $\I_{3,6}$ (with respect to $M_S$). If this condition holds, the following two Propositions show that the total number of weight vectors $w_S\in \trop\left(\Gr(3,n)\right)$ reduces significantly, thus reducing the total number of degenerations to study.

\begin{prop}\label{perm-equality}
    Let $(\beta_1,\beta_2,\beta_3)$ be a PBW sequence for $\Gr\left(3,4\right)$ and $\sigma\in S_3$. Consider the following iterated sequences for $\Gr\left(3,n\right)$
    \begin{align*}
        S & := \left(\epsilon_{i_1}-\epsilon_{n}, \epsilon_{i_2}-\epsilon_n,\epsilon_{i_3}-\epsilon_n,\dots,\beta_1,\beta_2,\beta_3\right)\\
        S_{\sigma} & := \left(\epsilon_{i_1}-\epsilon_n,\epsilon_{i_2}-\epsilon_n,\epsilon_{i_3}-\epsilon_n,\dots,\beta_{\sigma(1)},\beta_{\sigma(2)},\beta_{\sigma(3)}\right).
    \end{align*}
    
    Given  $e_S$ a projection that preserves inequalities with respect to $S$, there exists a projection $e_{S_{\sigma}}$ that preserves inequalities with respect to $S_{\sigma}$ and satisfies
    \begin{equation*}
        e_{S_{\sigma}}M_{S_{\sigma}} = e_{S}M_S.
    \end{equation*}
\end{prop}

The following remark from the proof of Proposition \ref{binomial} will be needed.

\begin{rem}\label{obs1}
    Let $R_{I,J}\neq 0$ be a Plücker relation and $s_1,s_2\in J$ be such that
    \begin{equation*}
        \V_S\left(\overline{p_{I\cup s_1}p_{J\setminus s_1}}\right)\prec_{\Psi_S}\V_S\left(\overline{p_{I\cup s_2}p_{J\setminus s_2}}\right).
    \end{equation*} By the proof of Proposition \ref{binomial}, the comparison of these two elements is firstly done in the first three entries; if these entries coincide, pass to the next triad; inductively, repeat until reaching the second-to-last triad, where the tie must break. 
\end{rem}

\begin{proof}[Proof of Proposition \ref{perm-equality}]
    Let $S$ and $S_{\sigma}$ be the two iterated sequences and $\sigma\in S_3$. Write the projection $e_S:\Z^{3(n-3)}\rightarrow \Z$ as follows
    \begin{equation*}
        e_S = \left(e_1, \dots, e_{3(n-1-3)},e_{3(n-1-3) + 1},e_{3(n-1-3) + 2},e_{3(n-1-3) + 3}\right).
    \end{equation*} Let $I\in I_{3,n}$ be a multiindex, $\overline{p_I}$ a Plücker coordinate, and $\mathbf{m}\in \Z^{3(n-1-3)}, \mathbf{n}\in \Z^3$ such that 
    \begin{equation*}
        \V_S\left(\overline{p_I}\right) = \left(\mathbf{m},\mathbf{n}\right).
    \end{equation*} Then
    \begin{equation*}
        \V_{S_{\sigma}}\left(\overline{p_I}\right) = \left(\mathbf{m},\sigma\mathbf{n}\right),
    \end{equation*} with $\sigma\mathbf{n}$ is obtained by permuting the entries of $\mathbf{n}$ using $\sigma$.

    Define $e_{S_{\sigma}}:\Z^{3(n-3)}\rightarrow \Z$ as
    \begin{equation*}
        e_{S_{\sigma}} := \left(e_1,\dots,e_{3(n-1-3)}, e_{3(n-1-3)+\sigma(1)},e_{3(n-1-3)+\sigma(2)},e_{3(n-1-3)+\sigma(3)}\right).
    \end{equation*} Notice that for all $I\in I_{3,n}$, the following equality holds:
    \begin{equation*}
        e_S\left(\V_S(\overline{p_I})\right) = e_{S_{\sigma}}\left(\V_{S_{\sigma}}(\overline{p_I})\right).
    \end{equation*} Now, for a Plücker relation $R_{I,J}\neq 0$ and $s_1,s_2\in J$ satisfying
    \begin{equation*}
        \V_S\left(\overline{p_{I\cup s_1}p_{J\setminus s_1}}\right) \prec_{\Psi_S} \V_S\left(\overline{p_{I\cup s_2}p_{J\setminus s_2}}\right),
    \end{equation*} by the Remark \ref{obs1} and the computation of valuations $\V_{S_{\sigma}}$, the following inequality holds
    \begin{equation*}
        \V_{S_{\sigma}}\left(\overline{p_{I\cup s_1}p_{J\setminus s_1}}\right) \prec_{\Psi_{S_{\sigma}}} \V_{S_{\sigma}}\left(\overline{p_{I\cup s_2}p_{J\setminus s_2}}\right).
    \end{equation*} Since $e_S$ and $e_{S_{\sigma}}$ are $\Z$-linear, the following inequality is true
    \begin{equation*}
        e_{S_{\sigma}} \left(\V_{S_{\sigma}}(\overline{p_{I\cup s_1}p_{J\setminus s_1}})\right) = e_S\left(\V_S(\overline{p_{I\cup s_1}p_{J\setminus s_1}})\right) < e_S\left(\V_S(\overline{p_{I\cup s_2}p_{J\setminus s_2}})\right) = e_{S_{\sigma}}\left(\V_{S_{\sigma}}(\overline{p_{I\cup s_2}p_{J\setminus s_2}})\right),
    \end{equation*} which proves that $e_{S_{\sigma}}$ preserves inequalities with respect to $S_{\sigma}$.
\end{proof}

At this point, recall \cite[Proposition 1]{Bossinger:2021}: given an iterated sequence $S$ for $\Gr\left(2,n\right)$, there exists a trivalent tree $T_S$ that satisfies $\In_{M_S}\left(\I_{2,n}\right) = \In_{T_S}\left(\I_{2,n}\right)$. This tree is the output of \cite[Algorithm 1]{Bossinger:2021}, which only depends on the first index of each iteration. Thus, it is natural to wonder if this the case for an iterated sequence for $\Gr\left(3,n\right)$. The answer is positive, and is the content of the next Proposition.

\begin{prop}\label{change-equality}
    Let $S'$ be an iterated sequence for $\Gr\left(3,n-1\right)$. Let $i_1,i_2\in [n-1]$ and $i,j\in [n-1]\setminus \lbrace i_1,i_2\rbrace$. Consider the following iterated sequences for $\Gr\left(3,n\right)$
    \begin{align*}
        S_1 & = \left(\epsilon_{i_1}-\epsilon_n,\epsilon_{i_2}-\epsilon_n,\epsilon_{i}-\epsilon_n,S'\right)\\
        S_2 & = \left(\epsilon_{i_1}-\epsilon_n,\epsilon_{i_2}-\epsilon_n,\epsilon_j-\epsilon_n,S'\right).
    \end{align*} 
    
    Given a projection $e_{S_1}$ that preserves inequalities with respect to $S_1$, there exists a projection $e_{S_2}$ that preserves inequalities with respect to $S_{2}$ and satisfies
    \begin{equation*}
        e_{S_1}M_{S_1} = e_{S_2}M_{S_2}.
    \end{equation*}
\end{prop}

There is a remark analogous to Remark \ref{obs1}.

\begin{rem}\label{obs2}
    Let $S=\left(\epsilon_{i_1}-\epsilon_n,\epsilon_{i_2}-\epsilon_n,\epsilon_{i_3}-\epsilon_n,\dots\right)$ be an iterated sequence for $\Gr\left(3,n\right)$, $R_{I,J}\neq 0$ a Plücker relation with $n\in I$ or $n\in J$, and $s_1,s_2\in J$ such that
    \begin{equation*}
        \V_S\left(\overline{p_{I\cup s_1}p_{J\setminus s_1}}\right) \prec_{\Psi_S} \V_S\left(\overline{p_{I\cup s_2}p_{J\setminus s_2}}\right).
    \end{equation*}
    By the proof of Proposition \ref{binomial}, the comparison of  these two elements is first done in the first two entries of the first triad; if they coincide, pass to the first two elements of the second triad; inductively, repeat until reaching the second-to-last triad where the tie must break.
\end{rem}

\begin{proof}[Proof of Proposition \ref{change-equality}]
    Let $S_1$ and $S_2$ be two iterated sequences for $\Gr\left(3,n\right)$. Write the projection $e_{S_1}:\Z^{3(n-3)}\rightarrow \Z$ as follows
    \begin{equation*}
        e_{S_1} = \left(e_1,e_2,e_3,\dots,e_{3(n-3)}\right).
    \end{equation*}
    Let $I=(i_1,i_2,n)\in I_{3,n}$ and $\overline{p_I}$ the Plücker coordinate. The valuations of $\overline{p_I}$ are
    \begin{align*}
        \V_{S_1}\left(\overline{p_I}\right) & = \left(f_3,\V_{S'}(\overline{p_{I\setminus n\cup i}})\right),\\
        \V_{S_2}\left(\overline{p_I}\right) & = \left(f_3,\V_{S'}(\overline{p_{I\setminus n\cup j}})\right).
    \end{align*} Notice that this is the only Plücker coordinate that satisfies
    \begin{equation*}
        \V_{S_1}\left(\overline{p_I}\right) \neq \V_{S_2}\left(\overline{p_I}\right).
    \end{equation*} Define $e_{S_2}:\Z^{3(n-3)}\rightarrow \Z$ as
    \begin{equation*}
        e_{S_2} :=\left(e_1,e_2,e_{S_1}(\V_{S_1}(\overline{p_I})-\V_{S_2}(\overline{p_{I\setminus n\cup j}})),\dots,e_{3(n-3)}\right).
    \end{equation*} Notice that for $I\in I_{3,n}$ with $I\neq (i_1,i_2,n)$, the equality holds
    \begin{equation*}
        e_{S_2}\left(\V_{S_2}(\overline{p_I})\right) = e_{S_1}\left(\V_{S_1}(\overline{p_I})\right).
    \end{equation*} For $I=(i_1,i_2,n)$, the equality holds as well:
    \begin{equation*}
        e_{S_2}\left(\V_{S_2}(\overline{p_I})\right) = e_{S_1}\left(\V_{S_1}(\overline{p_I})\right) - e_{S_2}\left(\V_{S_2}(\overline{p_{I\setminus n\cup j}})\right) + e_{S_2}\left(\V_{S_2}(\overline{p_{I\setminus n\cup j}})\right) = e_{S_1}\left(\V_{S_1}(\overline{p_I})\right).
    \end{equation*}
    Now, given a Plücker relation $R_{I,J}\neq 0$ with $s_1,s_2\in J$ such that
    \begin{equation*}
        \V_{S_1}\left(\overline{p_{I\cup s_1}p_{J\setminus s_1}}\right) \prec_{\Psi_{S_1}} \V_{S_1}\left(\overline{p_{I\cup s_2}p_{J\setminus s_2}}\right),
    \end{equation*} by Remark \ref{obs2} and the computation of the valuations, the following inequality holds
    \begin{equation*}
        \V_{S_2}\left(\overline{p_{I\cup s_1}p_{J\setminus s_1}}\right) \prec_{\Psi_{S_2}} \V_{S_2}\left(\overline{p_{I\cup s_2}p_{J\setminus s_2}}\right).
    \end{equation*} Since $e_{S_1}$ and $e_{S_2}$ are $\Z$-linear, the following inequality is satisfied
    \begin{equation*}
        e_{S_1}\left(\V_{S_1}(\overline{p_{I\cup s_1}p_{J\setminus s_1}})\right)  = e_{S_2}\left(\V_{S_2}(\overline{p_{I\cup s_1}p_{J\setminus s_1}})\right) < e_{S_1}\left(\V_{S_1}(\overline{p_{I\cup s_2}p_{J\setminus s_2}})\right) = e_{S_2}\left(\V_{S_2}(\overline{p_{I\cup s_2}p_{J\setminus s_2}})\right),
    \end{equation*} which proves that $e_{S_2}$ preserves inequatilies with respect to $S_2$.
\end{proof}

Consider from now on that every iterated sequence for $\Gr\left(3,n\right)$ is iterated from a PBW sequence for $\Gr\left(3,4\right)$. A direct computation shows that
\begin{equation*}
    \# \mathcal{S}_{3,n} = \prod_{l=0}^{n-4} \left(n-l-1\right)\left(n-l-2\right)\left(n-l-3\right).
\end{equation*}

Let $S_1 ,S_2\in \mathcal{S}_{3,n}$ and assume that the first two indices of every iteration coincide. Given a projection $e_{S_1}$ that preserves inequalities with respect to $S_1$, by Proposition \ref{change-equality} there exists a projection $e_{S_2}$ that preserves inequalities with respect to $S_2$ and satisfies
\begin{equation*}
    e_{S_1}M_{S_1} = e_{S_2}M_{S_2}.
\end{equation*} Furthermore, if $S_1$ is iterated from the PBW sequence $(\beta_1,\beta_2,\beta_3)$ and $S_2$ is iterated from a PBW sequence $(\beta_{\sigma(1)},\beta_{\sigma(2)},\beta_{\sigma(3)})$ with $\sigma\in S_3$, then by Proposition \ref{perm-equality}, the same equality holds. Thus, and by the proof of \cite[Lemma 3]{Bossinger_2020}, the following proposition is proved.

\begin{prop}\label{number_of_vectors}
    If $\left\lbrace R_{I,J}\neq 0:I\in I_{2,n},J\in I_{4,n} \right\rbrace$ is a reduced Gröbner basis of $\I_{3,n}$ with respect to $M_{S}$ for all $S\in \mathcal{S}_{3,n}$, then
    \begin{equation*}
        \#\left\lbrace \In_{M_S}(\I_{3,n}):S\in \mathcal{S}_{3,n}\right\rbrace \leq \prod_{l=0}^{n-5}\left(n-l-1\right)\left(n-l-2\right).
    \end{equation*}
\end{prop}

\subsection*{The Grassmannian \texorpdfstring{$\Gr(3,6)$}{Gr(3,6)}}
Let $S \in \mathcal{S}_{3,6}$. Let $R_{I,J}\neq 0$ be any Plücker relation and $e_S:\Z^{9}\rightarrow \Z$ a projection that preserves inequalities with respect to $S$. This means that for $s_1,s_2\in J$ it satisfies
\begin{equation*}
    \V_S\left(\overline{p_{I\cup s_1}p_{J\setminus s_1}}\right) \prec_{\Psi_S} \V_S\left(\overline{p_{I\cup s_2}p_{J\setminus s_2}}\right),
\end{equation*} then 
\begin{equation}\label{inequalities}
    e_S\left(\V_S(\overline{p_{I\cup s_1}p_{J\setminus s_1}})\right) < e_S\left(\V_S(\overline{p_{I\cup s_2}p_{J\setminus s_2}})\right).
\end{equation} Regarding the projection $e_S:\Z^9\rightarrow \Z$ as an element $e_S\in \Z^9$, the inequalities of the form \eqref{inequalities} define a polyhedral cone in $\Z^9$ (cf. \cite[Chapter 2]{sturmfels1996grobner}). By implementing Algorithm \ref{algorithm}, all the different inequalities obtained from the computation of initial forms of Plücker relations are calculated. In turn, by using \verb|polymake| (cf. \cite{polymake:2000}), this cone is computed and a vector in its relative interior is a projection $e_S$ that preserves inequalities with respect to $S$. For all projections $e_S$, define as before $w_S:=e_S M_S$.

Using the computations from the previous section, the total number of iterated sequences for $\Gr\left(3,6\right)$ is $\# \mathcal{S}_{3,6}=8640$. Repeating the process above for all different iterated sequences for $\Gr\left(3,6\right)$ yields the total number of weight vectors
\begin{equation}
    \#\left\lbrace w_S\in \R^{20} :S\in \mathcal{S}_{3,6} \right\rbrace = 240,
\end{equation} which was expected by the Proposition \ref{number_of_vectors}.

 The initial ideals $\In_{w_S}\left(\I_{3,6}\right)$ are computed using \verb|macaulay2| (cf. \cite{M2}), and the following facts are obtained
\begin{itemize}
    \item For all $S\in \mathcal{S}_{3,6}$, the minimum number of generators of $\In_{w_S}\left(\I_{3,6}\right)$ coincides with the minimum number of generators of $\I_{3,6}$. Thus, the set of Plücker relations is a reduced Gröbner basis for $\I_{3,6}$ with respect to $M_S$ and the next equality holds
    \begin{equation*}
        \In_{w_S}\left(\I_{3,6}\right) = \In_{M_S}\left(\I_{3,6}\right).
    \end{equation*}
    \item For all $S\in \mathcal{S}_{3,6}$, the ideal $\In_{w_S}\left(\I_{3,6}\right)$ is binomial and prime. By \cite[Lemma 1.1]{sturmfels1996equationsdefiningtoricvarieties}, $\In_{w_S}\left(\I_{3,6}\right)$ is toric, and thus defines a toric variety.
    \item If $S_1$ and $S_2$ are two iterated sequences that do not satisfy the hypotheses of Proposition \ref{change-equality} (this is, the first or second or both indices of an iteration differ), then 
    \begin{equation*}
        \In_{M_{S_2}}\left(\I_{3,6}\right) \neq \In_{M_{S_2}}\left(\I_{3,6}\right).
    \end{equation*} 
\end{itemize}

This proves the following Theorem.

\begin{thm}[First Classification]\label{prim-class}
    The assignment $\mathcal{S}_{3,6}\rightarrow \left\lbrace \textrm{initial ideals of }\I_{3,6} \right\rbrace$ given by
    \begin{equation*}
        S\mapsto \In_{M_S}\left(\I_{3,6}\right)
    \end{equation*} is an assignment
    \begin{equation*}
        \mathcal{S}_{3,6} \rightarrow \left\lbrace \textrm{toric initial ideals of }\I_{3,6} \right\rbrace
    \end{equation*} and the image has cardinality $\#\left\lbrace \In_{M_S}(\I_{3,6}):S\in \mathcal{S}_{3,6}\right\rbrace = 240$.
\end{thm}

A direct consequence of this Theorem is related to the \emph{value semi-group} $S\left(A_{3,6},\V_S\right)$ (cf. e.g. \cite[$\S 2.1$]{Bossinger:2021}).

\begin{cor}
    For all $S\in \mathcal{S}_{3,6}$, the set $\left\lbrace \overline{p_I}:I\in I_{3,6}\right\rbrace$ forms a Khovanskii basis of $\left(A_{3,6},\V_S\right)$.
\end{cor}
\begin{proof}
    Let $S\in \mathcal{S}_{3,6}$. By the proof of \cite[Theorem 1]{Bossinger:2021}, the valuation $\V_S:A_{3,6}\setminus \lbrace 0\rbrace \rightarrow \left(\Z^{9},\preceq_{\Psi_S}\right)$ is a full-rank valuation. By Theorem \ref{prim-class}, the initial ideal $\In_{M_S}\left(\I_{3,6}\right)$ is prime. By \cite[Theorem 1]{Bossinger_2020}, the value semi-group $S\left(A_{3,6},\V_S\right)$ is generated by $\left\lbrace\V_S( \overline{p_I}):I\in I_{3,6}\right\rbrace$.
\end{proof}

By Theorem \ref{prim-class}, the notation can be simplified: consider an iterated sequence 
\begin{equation*}
    S = \left(\epsilon_{i_1}-\epsilon_6,\epsilon_{i_2}-\epsilon_6,\epsilon_{i_3}-\epsilon_6,\epsilon_{j_1}-\epsilon_5,\epsilon_{j_2}-\epsilon_5,\epsilon_{j_3}-\epsilon_5,\beta_1,\beta_2,\beta_3\right).
\end{equation*} Since the initial ideal $\In_{M_S}\left(\I_{3,6}\right)$ only depends on the first two indices of each iteration, the initial ideal will be written 
\begin{equation}
    I_{(i_1,i_2;j_1,j_2)} := \In_{M_S}\left(\I_{3,6}\right).
\end{equation}

The next step in the classification of the initial ideals considers the action of a subgroup of $\operatorname{Aut}\left(\I_{3,6}\right)$. Consider simple transpositions $s_i=\begin{pmatrix} i & i+1\end{pmatrix}\in S_6$. The action of $s_i$ on $[6]$ can be extended to $I_{3,6}$ as $s_i.(i_1,i_2,i_3)= (s_i(i_1), s_i(i_2), s_i(i_3))$. They can further be extended to Plücker variables as
\begin{equation*}
    s_i.p_{(i_1,i_2,i_3)} = \left\lbrace \begin{matrix} p_{s_i.(i_1,i_2,i_3)} & \textrm{if } (i_1,i_2) \neq (i,i+1) \textrm{ and } (i_2,i_3) \neq (i,i+1)\\
    -p_{(i_1,i_2,i_3)} &  \textrm{else}
    \end{matrix}\right. ,
\end{equation*} and finally extended as an isomorphism of $\I_{3,6}$. Consider the following subgroup of $\operatorname{Aut}\left(\I_{3,6}\right)$
\begin{equation*}
    G:=\left\langle s_i:i\in [5] \right\rangle_{\operatorname{Aut}\left(\I_{3,6}\right)}.
\end{equation*} Under the action of this group, two ideals $\In_{M_{S_1}}\left(\I_{3,6}\right)$ and $\In_{M_{S_2}}\left(\I_{3,6}\right)$ are equivalent if there exists $g\in G$ such that $\In_{M_{S_1}}\left(\I_{3,6}\right) = g\left(\In_{M_{S_1}}(\I_{3,6})\right)$; an equivalence class is called a $G$-orbit. Notice that this means that two ideals in the same $G$-orbit are isomorphic, and thus, the toric varieties they define are isomorphic. These isomorphisms are written in \verb|macaulay2| to compute the $G$-orbits and they yield the following classification up to the $G$-action.

\begin{thm}[Partial Classification]\label{second_class}
    Consider $\mathbb{I}_{3,6} := \lbrace \In_{M_S}(\I_{3,6}):S\in \mathcal{S}_{3,6}\rbrace$. The $G$-orbits of $\mathbb{I}_{3,6}$ intersected with $\mathbb{I}_{3,6}$ are
    \begin{align*}
    O_{1}\cap \mathbb{I}_{3,6} &:= \left\lbrace I_{(s_1,s_2;r,k)} : k\in [4]  \wedge (\lbrace s_1,s_2\rbrace = \lbrace r,5 \rbrace \textrm{ or } \lbrace s_1,s_2\rbrace = [5]\setminus \lbrace r,k,M_k  \rbrace ) \right\rbrace,\\
    O_{2} \cap \mathbb{I}_{3,6}&:= \left\lbrace I_{(k,s_1;s_2,k)} :k\in [4], s_1 \in [5]\setminus \lbrace k\rbrace \, \wedge\, s_2 \in [4]\setminus \lbrace k\rbrace  \right\rbrace,\\
    O_{3} \cap \mathbb{I}_{3,6}&:= \left\lbrace I_{(s_1,k;s_2,k)} : s_1 \in [5]\setminus \lbrace k\rbrace \, \wedge\, s_2 \in [4]\setminus \lbrace k\rbrace  \right\rbrace,\\
    O_4 \cap \mathbb{I}_{3,6} &:= \mathbb{I}_{3,6} \setminus \left(\bigcup_{i=1,2,3} O_i\right).
\end{align*} These intersections have the following cardinalities and are listed with its corresponding isomorphism class as described in \cite[$\S5$]{speyer2003tropicalgrassmannian}
\begin{table}[H]
    \centering
    \begin{tabular}{|c|c|c|}
        \hline
        Orbit & $\#$ & Isomorphism class \\
        \hline
        $O_1\cap\mathbb{I}_{3,6}$ & $48$ & $ EEFF1 $\\
        \hline
        $O_2\cap\mathbb{I}_{3,6}$ & $48$ & $EFFG$ \\
        \hline 
        $O_3\cap\mathbb{I}_{3,6}$ & $48$ & $EEFF2$\\
        \hline
        $O_4\cap\mathbb{I}_{3,6}$ & $96$ & $EEFG$\\
        \hline
    \end{tabular}
    \caption{Classification of the $G$-orbits in $\mathbb{I}_{3,6}$}
    \label{isom-class}
\end{table}
\end{thm}

\begin{rem}
    There exist initial ideals $\In_{M_S}\left(\I_{3,6}\right)$ whose image under one of the simple transpositions is not contained in $\mathbb{I}_{3,6}$. Therefore, the intersections $O\cap \mathbb{I}_{3,6}$ in the previous theorem are necessary. 
\end{rem}

\subsection*{Toric degenerations}
The main results of the previous sections can be briefly restated as follows: for every iterated sequence $S\in \mathcal{S}_{3,6}$, the ideal $\In_{M_S}\left(\I_{3,6}\right)$ is toric and up to strict equality, there are 240 different toric initial ideals labeled by the first two indices of each iteration (Theorem \ref{prim-class}); up to the action of $S_6\leq \operatorname{Aut}\left(\I_{3,6}\right)$, there are four different orbits (Theorem \ref{second_class}); equivalently, the set 
\begin{equation*}
    \mathbb{V}_{3,6} := \left\lbrace \operatorname{Spec}(\C[p_K:K\in I_{3,6}]/\In_{M_S}(\I_{3,6})):S\in \mathcal{S}_{3,6}\right\rbrace
\end{equation*}    
contains, up to isomorphism induced by $S_6$, four different toric varieties, classified according to Theorem \ref{second_class}. As was mentioned in the Introduction, one of the main reasons to study birational sequences and, in particular, iterated sequences, is to construct toric degenerations arising as Gröbner degenerations. The description of this construction can be found, for example, in \cite{Bossinger:2021,bossinger2023surveytoricdegenerationsprojective}. It will be briefly sketched here in the case of the Grassmannian $\Gr\left(3,6\right)$.

Let $w\in \trop\left(\Gr(3,6)\right)$ be an arbitrary weight vector and $\I_{3,6}$ the Plücker ideal.  For $t\in \C$, consider the following family of ideals
\begin{equation*}
    \Tilde{I}_t := \left\langle   t^{-\min_{u}\lbrace u\cdot w \rbrace}f\left(t^{w_{I_1}}p_{I_1},\dots,t^{w_{I_N}}p_{I_N}\right): f=\sum a_u p^u\in \I_{3,6}\right\rangle \subset \C\left[t,p_K^{\pm 1}:K\in I_{3,6}\right].
\end{equation*}  This describes a flat family over $\C$ (cf. \cite[Section 15.8]{Eisenbud:CommAlg}). Consider the following three varieties
\begin{equation*}
    \operatorname{Spec}\left(A_{3,6}\right), \quad \operatorname{Spec}\left(\C\left[t,p_K:K\in I_{3,6}\right]/\Tilde{I}_t\right),\quad \operatorname{Spec}\left(\C[p_K:K\in I_{3,6}]/\In_w(\I_{3,6})\right).
\end{equation*} Write $I_s := \left.\Tilde{I}_t\right|_{t=s}$. If $s\neq 0$, there is an automorphism of $\C\left[p_K:K\in I_{3,6}\right]$ sending $I_s$ to $\I_{3,6}$. Then $\operatorname{Spec}\left(\C\left[t,p_K:K\in I_{3,6}\right]/\Tilde{I}_t\right)$ defines a degeneration of $\operatorname{Spec}(A_{3,6})$, called the \emph{Gröbner degeneration}. If the ideal $\In_w\left(\I_{3,6}\right)$ is toric, then the degeneration is called \emph{toric}.

The consequence of Theorems \ref{prim-class} and \ref{second_class} is the following corollary.

\begin{cor}\label{class_toric}
    Every iterated sequence $S\in \mathcal{S}_{3,6}$ induces a toric degeneration of $\Gr\left(3,6\right)$. These degenerations are, up to the action of $S_6\leq \operatorname{Aut}\left(\I_{3,6}\right)$, classified according to the classification in Theorem \ref{second_class}.
\end{cor}

This paper concludes by noticing that the equivalences presented in \cite[Remark 1]{Bossinger:2021} for $\Gr\left(2,n\right)$ do not generalize when considering the Grassmannian $\Gr\left(3,6\right)$. By \cite[$\S 5$]{speyer2003tropicalgrassmannian}, there are 7 isomorphism classes of maximal cones in $\trop\left(\Gr(3,6)\right)$. By Theorem \ref{second_class}, only four of these classes ($EEFF1,EEFF2,EFFG,EEFG$) correspond to initial ideals induced by iterated sequences, so the function $\In_{M_S}\left(\I_{3,6}\right)\mapsto \In_{w_S}\left(\I_{3,6}\right)$ is injective but cannot be surjective. This is, there is no ``natural" arrow pointing left in the next diagram
\begin{equation*}
    \left\lbrace \begin{matrix}\textrm{toric degenerations of }\Gr(3,6)\\\textrm{induced by iterated sequences}\end{matrix}\right\rbrace \rightarrow \left\lbrace \begin{matrix}\textrm{toric degenerations of }\Gr(3,6)\\\textrm{induced by }\trop(\Gr(3,6))\end{matrix}\right\rbrace.
\end{equation*}   Furthermore, the classes $EEFF1,EEFF2,EFFG,EEFG,EEEG$, together with a class corresponding to an edge of the form $GG$, correspond to degenerations induced by plabic graphs (cf. \cite[Table 1]{bossinger2016toricdegenerationsgr2ngr36}), so there is an injective function 
\begin{equation*}
    \left\lbrace \begin{matrix}\textrm{toric degenerations of }\Gr(3,6)\\\textrm{induced by iterated sequences}\end{matrix}\right\rbrace \rightarrow \left\lbrace \begin{matrix}\textrm{toric degenerations of }\Gr(3,6)\\\textrm{induced by plabic graphs}\end{matrix}\right\rbrace.
\end{equation*} which cannot be surjective since there is no initial ideal $\In_{M_S}\left(\I_{3,n}\right)$ corresponding to the class $EEEG$ or to the class corresponding to an edge of the form $GG$.

\bibliographystyle{alpha}
\bibliography{biblio}

J. TORRES HENESTROZA: Facultad de Ciencias, Universidad Nacional Autónoma de México, Investigación Científica, C.U., Alcaldía Coyoacán, C.P. 04510. Ciudad de México, México.

\emph{E-mail address}: \verb|joaquinth@ciencias.unam.mx|

\end{document}